\documentclass{article}
\usepackage{amsmath}
\usepackage{amssymb}
\usepackage{amsthm}
\usepackage{amsfonts, url}
\usepackage{graphicx}
\usepackage[colorlinks,urlcolor=blue]{hyperref}
\theoremstyle{definition}
\newtheorem{definition}{Definition}[section]
\newtheorem{example}[definition]{Example}
\newtheorem{remark}[definition]{Remark}
\newtheorem{question}[definition]{Question}
\theoremstyle{plain}
\newtheorem{theorem}[definition]{Theorem}

\newtheorem{proposition}[definition]{Proposition}

\numberwithin{equation}{section}
\begin{document}
%%%%%%%%%%%%%%%%%%%%%%%%%%%%%%%%%%%%%%%%%%%%%%%%%%%%%%%%%%%%%%%%%%%%%%

%%%\\\ %%%%%%%%%%%%%%%%%%%%%%%%%%%%%%%%%%%%%%%%%%%%%%%%%%%%%%%%%%%%%%%%
\newcommand{\LomX}{X}   \newcommand{\LomY}{Y}
\def\LomA{F}
\def\op{*}   \def\opX{\op_{\LomX}}  \def\opY{\op_{\LomY}}
\def\leX{\leq_{\LomX}}   \def\leY{\leq_{\LomY}}
\def\kx{\mathfrak a}   \def\ky{\mathfrak b}  \def\kz{\mathfrak c}
\def\kJ{0}   \def\kJX{\kJ_{\LomX}}  \def\kJY{\kJ_{\LomY}}
\def\Xo0{{\mathcal \LomX}}   \def\Yo0{{\mathcal \LomY}}
\def\kXo0{$\Xo0$ $:=$ $(\LomX,$ $\opX,$ $\kJX)$}
\def\kYo0{$\Yo0$ $:=$ $(\LomY,$ $\opY,$ $\kJY)$}
\def\x{x}  \def\y{y} \def\z{z}
\def\jt{t}   \def\js{s}  \def\kt{\beta}  \def\vt{\varepsilon}
%%%\\\ %%%%%%%%%%%%%%%%%%%%%%%%%%%%%%%%%%%%%%%%%%%%%%%%%%%%%%%%%%%%%%%%
%%%%%%%%%%%%%%%%%%%%%%%%%%%%%%%%%%%%%%%%%%%%%%%%%%%%%%%%%%%%%%%%%%%%%%%
 \def\xx{\x\opX \x}  \def\yy{\y\opX \y} \def\zz{\z\opX \z}
\def\xy{\x\opX \y}  \def\yx{\y\opX \x} \def\yz{\y\opX \z}
\def\zx{\z\opX \x}  \def\xz{\x\opX \z}
\newcommand{\kxy}{\kx\opX \ky}  \newcommand{\kyy}{\ky\opX \ky}
\newcommand{\kyz}{\ky\opX \kz}  \newcommand{\kzx}{\kz\opX \kx}
\newcommand{\kxz}{\kx\opX \kz}
%%%%%%%%%%%%%%%%%%%%%%%%%%%%%%%%%%%%%%%%%%%%%%%%%%%%%%%%%%%%%%%%%%%%%%%
%%%\\\ %%Fuzzy ==> %%%%%%%%%%%%%%%%%%%%%%%%%%%%%%%%%%%%%%%%%%%%%%%%%%%%
\newcommand{\kap}{\hslash}  \newcommand{\kbp}{\eth}
\def\kwp{f} \def\kmp{\mu}
%%%\\\ %%<== Fuzzy %%%%%%%%%%%%%%%%%%%%%%%%%%%%%%%%%%%%%%%%%%%%%%%%%%%%
%%%\\\ %%%Examples ==>%%%%%%%%%%%%%%%%%%%%%%%%%%%%%%%%%%%%%%%%%%%%%%%%%
\newcommand{\0}{\varrho_0}  \newcommand{\1}{\varrho_1} \newcommand{\2}{\varrho_2}
\newcommand{\3}{\varrho_3}  \newcommand{\4}{\varrho_4}
%%%\\\ %%%<== Examples%%%%%%%%%%%%%%%%%%%%%%%%%%%%%%%%%%%%%%%%%%%%%%%%%

%%%%%%%%%%%%%%%%%%%%%%%%%%%%%%%%%%%%%%%%%%%%%%%%%%%%%%%%%%%%%%%%%%%%%%
\noindent {\large \textbf
  	{Fuzzy normed BCK-algebras and BCI-algebras}}  \\[13pt]
%%%%%%%%%%%%%%%%%%%%%%%%%%%%%%%%%%%%%%%%%%%%%%%%%%%%%%%%%%%%%%%%%%%%%%%
{Young Bae Jun$^{1}$ and Ravikumar Bandaru$^{2,*}$} \\[2mm]
%%%%%%%%%%%%%%%%%%%%%%%%%%%%%%%%%%%%%%%%%%%%%%%%%%%%%%%%%%%%%%%%%%%%%%
Department of Mathematics Education, \\
Gyeongsang National University,
Jinju 52828, Korea \\
e-mail: skywine@gmail.com  \\
ORCID iD \url{https://orcid.org/0000-0002-0181-8969} \\ [2mm]
%%%%%%%%%%%%%%%%%%%%%%%%%%%%%%%%%%%%%%%%%%%%%%%%%%%%%%%%%%%%%%%%%%%%%%%
$^{1}$Department of Mathematics, School of Advanced Sciences,
VIT-AP University,  Beside AP Secretariate, Andhra Pradesh-522237, India  \\
e-mail: ravimaths83@gmail.com 	\\
ORCID iD  \url{https://orcid.org/0000-0001-8661-7914}        \\ [2mm]
%%%%%%%%%%%%%%%%%%%%%%%%%%%%%%%%%%%%%%%%%%%%%%%%%%%%%
 $^*$ Correspondence: Ravikumar Bandaru (ravimaths83@gmail.com)

\noindent\hrulefill \\
%%%%%%%%%%%%%%%%%%%%%%%%%%%%%%%%%%%%%%%%%%%%%%%%%%%%%%%%%%%%%%%%%%%%%%
 {\bf Abstract}
%%%%%%%%%%%%%%%%%%%%%%%%%%%%%%%%%%%%%%%%%%%%%%%%%%%%%%%%%%%%%%%%%%%%%%
	   In this paper, we introduce and study the notion of fuzzy normed
	   BCK-algebras and fuzzy normed BCI-algebras as a natural extension of
	   classical normed algebraic structures into the fuzzy setting.
	   A fuzzy norm on a BCK/BCI-algebra is defined as a mapping from the
	   algebra and a positive real parameter into the unit interval satisfying
	   suitable axioms analogous to those of fuzzy normed linear spaces.
	   Several examples are presented to illustrate the validity  of the axioms.
Fundamental properties of fuzzy normed BCK/BCI-algebras are established,
	   including monotonicity, chained triangle inequalities, and order-related
	   behaviors. It is shown that every BCK/BCI-algebra admits a fuzzy norm,
	   and the behavior of fuzzy norms under algebra homomorphisms is
	   investigated. Necessary and sufficient conditions are obtained for the
	   transfer of fuzzy norms via injective, surjective, and bijective
	   homomorphisms. A characterization theorem is proved showing that the
	   main fuzzy norm inequality can be reduced to a simpler condition.
	   These results generalize known concepts in fuzzy algebra and provide a
	   new analytical framework for studying BCK/BCI-algebras.
	
\vspace{2mm} \noindent
%%%%%%%%%%%%%%%%%%%%%%%%%%%%%%%%%%%%%%%%%%%%%%%%%%%%%%%%%%%%%%%%%%%%%%%

%%%%%%%%%%%%%%%%%%%%%%%%%%%%%%%%%%%%%%%%%%%%%%%%%%%%%%%%%%%%%%%%%%%%%%%

 \vspace{2mm} \noindent

 \noindent\hrulefill \\

%%%%%%%%%%%%%%%%%%%%%%%%%%%%%%%%%%%%%%%%%%%%%%%%%%%%%%%%%%%%%%%%%%%%%%
\section{Introduction}
%%%%%%%%%%%%%%%%%%%%%%%%%%%%%%%%%%%%%%%%%%%%%%%%%%%%%%%%%%%%%%%%%%%%%%

Fuzzy set theory, introduced by Zadeh \cite{Zadeh}, has provided a powerful
framework for modeling uncertainty and gradual membership in mathematics
and its applications. Since then, fuzzy concepts have been extensively
applied to algebraic systems such as groups, rings, vector spaces, and
logical algebras. In particular, fuzzy normed linear spaces have been
investigated by several authors as a natural extension of classical normed
spaces; see, for example, Felbin \cite{FSS48-239} and Bag and Samanta
\cite{AFMI-2013}.

BCK-algebras and BCI-algebras arise naturally in the algebraic study of
implication and non-classical logic. These structures were systematically
developed and studied in the monographs of Meng and Jun \cite{ MJ-Book}
and Huang \cite{H-Book}. Since their introduction, BCK/BCI-algebras have
been widely investigated with respect to ideals, filters, homomorphisms,
and related algebraic properties. Fuzzy concepts have also been introduced
into BCK-algebras, notably through the study of fuzzy ideals and related
structures (see, for instance, Jun and Kim \cite{JunKim}).

Motivated by the successful interaction between fuzzy theory and algebraic
logic, it is natural to extend the notion of fuzzy norms to BCK/BCI-algebras.
Although fuzzy norms have been studied in vector spaces and other algebraic
systems \cite{FSS48-239, AFMI-2013}, systematic research on fuzzy normed
BCK/BCI-algebras remains limited. This observation provides the main
motivation for the present work.

The purpose of this paper is to introduce the concept of fuzzy normed
BCK-algebras and fuzzy normed BCI-algebras and to investigate their
fundamental properties. We establish several inequalities characterizing
fuzzy norms, analyze their behavior with respect to the induced order, and
study their stability under algebra homomorphisms. In particular, we show
that every BCK/BCI-algebra admits a fuzzy norm and obtain necessary and
sufficient conditions for transferring fuzzy norms via injective,
surjective, and bijective homomorphisms.

The paper is organized as follows. In Section~2, we recall basic definitions
and properties of BCK/BCI-algebras and fuzzy concepts needed throughout the
paper. Section~3 introduces fuzzy normed BCK/BCI-algebras and develops their
fundamental properties through propositions, examples, and characterization
theorems. Finally, Section~4 summarizes the main results and outlines
directions for future research.

%%%%%%%%%%%%%%%%%%%%%%%%%%%%%%%%%%%%%%%%%%%%%%%%%%%%%%%%%%%%
\section{Preliminaries}
%%%%%%%%%%%%%%%%%%%%%%%%%%%%%%%%%%%%%%%%%%%%%%%%%%%%%%%%%%%%

In this section, we will provide some formal definitions and properties of
BCK/BCI-algebras and fuzzy concepts that will be used throughout the paper.

If a set $\LomX$ has a special element $\kJX$ and a binary operation $\opX$
satisfying the conditions:
\begin{itemize}
	\item[\rm ($I_1$)] $(\forall \kx,\ky,\kz\in \LomX)$ $(((\kx \opX \ky)\opX (\kx\opX \kz))\opX (\kz\opX \ky)=\kJX),$
	\item[\rm ($I_2$)] $(\forall \kx,\ky\in \LomX)$ $((\kx\opX (\kx \opX \ky))\opX \ky=\kJX),$
	\item[\rm ($I_3$)] $(\forall \kx \in \LomX)$ $(\kx\opX \kx=\kJX),$
	\item[\rm ($I_4$)] $(\forall \kx,\ky\in \LomX)$ $(\kx \opX \ky=\kJX, \, \ky\opX \kx=\kJX ~\Rightarrow ~\kx=\ky)$,
\end{itemize}
then we say that $\LomX$ is a  \textcolor{red}{\it BCI-algebra}.
If a BCI-algebra $\LomX$ satisfies the following identity:
\begin{itemize}
	\item[\rm (K)] $(\forall \kx \in \LomX)$ $(\kJX\opX \kx=\kJX),$
\end{itemize}
then $\LomX$ is called a \textcolor{red}{\it BCK-algebra.}

The order relation $\leX$ in a BCK/BCI-algebra $\Xo0$ is defined as follows:
\begin{align*}
	(\forall \kx, \ky\in \LomX) (\kx\leX \ky ~\Leftrightarrow ~\kx\opX \ky=\kJX).
\end{align*}

Every BCK/BCI-algebra $\LomX$ satisfies the following conditions (see \cite{H-Book, MJ-Book}):
\begin{align}
	& \label{{a1}}
	(\forall \kx \in \LomX)\left(\kx\opX \kJX=\kx\right),
	\\& \label{{a2}}
	(\forall \kx,\ky,\kz\in \LomX)\left(\kx\leX \ky \, \Rightarrow \, \kx\opX \kz\leX \ky\opX \kz, \, \kz\opX \ky\leX \kz\opX \kx\right),
	\\& \label{{a3}}
	(\forall \kx,\ky,\kz\in \LomX)\left((\kx \opX \ky)\opX \kz=(\kx\opX \kz)\opX \ky\right).
\end{align}
%%%%%%%%%%%%%%%%%%%%%%%%%%%%%%%%%%%%%%%%%%%%%%%%%%%%%%%%%%%%%%%%%%%

 Every BCI-algebra $\LomX$ satisfies (see \cite{H-Book}):
 \begin{align}
 	& \label{{b1}}
 	(\forall \kx,\ky\in \LomX)\left(\kx \opX (\kx\opX (\kx\opX \ky))=\kx\opX \ky\right),
 	\\& \label{{b2}}
 	(\forall \kx,\ky\in \LomX)\left(\kJX\opX (\kx\opX \ky)
 	=(\kJX\opX \kx)\opX (\kJX\opX \ky)\right).
 \end{align}
%%%%%%%%%%%%%%%%%%%%%%%%%%%%%%%%%%%%%%%%%%%%%%%%%%%%%%%%%%%%%%%%%%%

For more information on BCI-algebra and BCK-algebra,
please refer to the books \cite{H-Book, MJ-Book}.

%%%%%%%%%%%%%%%%%%%%%%%%%%%%%%%%%%%%%%%%%%%%%%%%%%%%%%%%%%%%%%%%%%%%%%
\section{Fuzzy normed BCK/BCI-algebras }
%%%%%%%%%%%%%%%%%%%%%%%%%%%%%%%%%%%%%%%%%%%%%%%%%%%%%%%%%%%%%%%%%%%%%%

In what follows, let \kXo0 (or simply $\Xo0$) be a BCK-algebra or a BCI-algebra
unless otherwise specified.

\begin{definition}\label{DqD41-251207}
 By a \textcolor{red}{\it fuzzy normed BCK/BCI-algebra} on $\Xo0$ we mean a pair
	$(\Xo0,$ $\| \cdot \|_{f})$ 	equipped with a mapping
	$\| \cdot \|_{f} : \LomX\times (0, \infty)\rightarrow [0,1]$ satisfying the following conditions
	for all $\x, \y, \z\in \LomX$ and $\jt, \js \in (0, \infty)$:
	\begin{align}
		&\label{(FN1)}  \|(\x,\jt)\|_{f}=1 \Leftrightarrow \x=\kJX,
		\\&\label{(FN2)}  \jt\geq \js ~\Rightarrow ~\|(\x, \jt)\|_{f}\geq \|(\x, \js)\|_{f},
		\\&\label{(FN3)}  \underset{\jt\to \infty} \lim \|(\x, \jt)\|_{f}=1,
		\\&\label{(FN4)}  \|(\xz, \js+ \jt)\|_{f}\geq \min \{\|(\xy, \js)\|_{f}, \|(\yz, \jt)\|_{f}\}.
	\end{align}
	We say that $\| \cdot \|_{f}$ is the \textcolor{red}{\it fuzzy BCK/BCI-norm} on $\Xo0$.
\end{definition}

\begin{example}\label{Eq[A2-1]}
(i) Consider a BCK-algebra $\Xo0 :=(\LomX, \op, \0)$ where $\LomX =\{\0, \1\}$ and the operation $\op$ is given by the following Cayley table:
\begin{align*}\begin{array}{c|cc}
		\op & \0 & \1  \\
		\hline
		\0 & \0 & \0  \\
		\1 & \1 & \0  \\
\end{array}\end{align*}
Define a mapping
\begin{align*}
	\| \cdot \|_{f} : \LomX\times (0, \infty)\rightarrow [0,1],
  ~(\x, \jt)\mapsto \begin{cases}
	1   & \text{\rm if $\x=\0$},\\
	1- \tfrac{1}{e^{t}} & \text{\rm if $\x =\1$}.
\end{cases}
\end{align*}
It is routine to verify that $\| \cdot \|_{f}$ is a fuzzy BCK-norm on $\Xo0$, so
$(\Xo0,$ $\| \cdot \|_{f})$ is a fuzzy normed BCK-algebra.

(ii)  Consider a BCK-algebra $\Xo0 :=(\LomX, \opX, 0)$ where $\LomX =\mathbb{N}\cup \{0\}$ and the operation $\opX$ is defined  by
\begin{align*}
 (\forall \x, \y \in \LomX) (\xy =\max \{0, \x -\y\}).
\end{align*}
Define a mapping
\begin{align*}
	\| \cdot \|_{f} : \LomX\times (0, \infty)\rightarrow [0,1],
	~(\x, \jt)\mapsto \begin{cases}
		1   & \text{\rm if $\x=0$},\\
		\tfrac{\jt}{\x +\jt} & \text{\rm if $\x \neq 0$}.
	\end{cases}
\end{align*}
It is routine to verify that $\| \cdot \|_{f}$ is a fuzzy BCK-norm on $\Xo0$, so
$(\Xo0,$ $\| \cdot \|_{f})$ is a fuzzy normed BCK-algebra.
\end{example}

\begin{example}\label{Eq[B2-1]}
(i) Consider a BCI-algebra $\Xo0 :=(\LomX, \op, \0)$ where $\LomX =\{\0, \1\}$ and the operation $\op$ is given by the following Cayley table:
\begin{align*}\begin{array}{c|cc}
		\op & \0 & \1  \\
		\hline
		\0 & \0 & \1  \\
		\1 & \1 & \0  \\
\end{array}\end{align*}
Define a mapping
\begin{align*}
	\| \cdot \|_{f} : \LomX\times (0, \infty)\rightarrow [0,1],
	~(\x, \jt)\mapsto \begin{cases}
		1   & \text{\rm if $\x=\0$},\\
		1-\tfrac{1}{e^{\jt}} & \text{\rm if $\x =\1$}.
	\end{cases}
\end{align*}
It is routine to verify that $\| \cdot \|_{f}$ is a fuzzy BCI-norm on $\Xo0$, so
$(\Xo0,$ $\| \cdot \|_{f})$ is a fuzzy normed BCI-algebra.

(ii) Consider a BCI-algebra $\Xo0 :=(\LomX, \opX, 0)$ where $\LomX =\mathbb{Z}$ (integers) and the operation $\opX$ is defined  by
$\xy =\x -\y$ for all $\x, \y \in \LomX$.
Define a mapping
\begin{align*}
	\| \cdot \|_{f} : \LomX\times (0, \infty)\rightarrow [0,1],
	~(\x, \jt)\mapsto \begin{cases}
		1   & \text{\rm if $\x=0$},\\
		\tfrac{\jt}{|\x| +\jt} & \text{\rm if $\x \neq 0$}.
	\end{cases}
\end{align*}
It is routine to verify that $\| \cdot \|_{f}$ is a fuzzy BCI-norm on $\Xo0$, so
$(\Xo0,$ $\| \cdot \|_{f})$ is a fuzzy normed BCI-algebra.
\end{example}

It is clear that every fuzzy normed BCK-algebra is a fuzzy normed BCI-algebra, but not vice versa.

\begin{proposition}\label{PqP47-251206}
Every fuzzy normed BCK-algebra $(\Xo0,$ $\| \cdot \|_{f})$ satisfies the following assertions
for all $\x, \y, \z\in \LomX$ and $\jt, \js \in (0, \infty)$:
\begin{align}
&\label{qcP47-251206-1} \|(\kJX, \jt)\|_{f} =1,
\\&\label{qcP47-251206-1a} \|(\x\opX \kJX, \jt)\|_{f} =\|(\x, \jt)\|_{f},
\\&\label{qcP47-251206-2} \|(\xx, \jt)\|_{f} =1,
\\&\label{qcP47-251206-3}  \x\neq \kJX ~\Rightarrow ~\|(\x, \jt)\|_{f} <1,
\\&\label{qcP47-251206-5}
   \x\leX \y ~\Rightarrow ~\|(\x, \js +\jt)\|_{f} \geq \|(\y, \jt)\|_{f},
\\&\label{qcP47-251206-6}  \|(\xz, \js +\jt)\|_{f} \geq \|(\x, \js)\|_{f},
\\&\label{qcP47-251206-7}
   \|(\x, \js +\jt)\|_{f} \geq \min \{\|(\xy, \js)\|_{f}, \|(\y, \jt)\|_{f}\},
\\&\label{qcP47-251206-8}
      \underset{\jt \to 0^+} \lim \|(\x, \jt)\|_{f} =1 ~\Rightarrow ~\x=\kJX.
\end{align}
\end{proposition}

\begin{proof}
The condition \eqref{(FN1)} forces $\|(\kJX, \jt)\|_{f} =1$.
The identity \eqref{{a1}} induces \eqref{qcP47-251206-1a}.
The combination of ($I_3$) and \eqref{qcP47-251206-1} induces \eqref{qcP47-251206-2}.
If $\x\neq \kJX$, then $\|(\x, \jt)\|_{f}\neq 1$. The combination of \eqref{(FN3)} and the range in $[0,1]$ yields to $\|(\x, \jt)\|_{f} <1$.
 Let $\x, \y\in \LomX$ be such that $\x\leX \y$. Then $\xy =\kJX$.
	Setting $\z=\kJX$ in \eqref{(FN4)} induces
\begin{align*}
	\|(\x, \js+ \jt)\|_{f} &\stackrel{\eqref{{a1}}}{=}
\|(\x \opX \kJX, \js+ \jt)\|_{f}
\\&\geq \min \{\|(\xy, \js)\|_{f}, \|(\y\opX \kJX, \jt)\|_{f}\}
 \\&\stackrel{\eqref{{a1}}}{=} \min \{\|(\kJX, \js)\|_{f}, \|(\y, \jt)\|_{f}\}
 \\&\stackrel{\eqref{(FN1)}}{=} \min \{1, \|(\y, \jt)\|_{f}\}
 =\|(\y, \jt)\|_{f},
\end{align*}	
so \eqref{qcP47-251206-5} holds.
  If we take $\y=\kJX$ in \eqref{(FN4)}, then
 \begin{align*}
 \|(\xz, \js +\jt)\|_{f}  &\geq
 \min \{\|(\x\opX \kJX, \js)\|_{f}, \|(\kJX\opX \z, \jt)\|_{f}\}
 \\& \stackrel{(K) \, \& \, \eqref{{a1}}}{=}
 \min \{\|(\x, \js)\|_{f}, \|(\kJX, \jt)\|_{f}\}
 \\&\stackrel{\eqref{(FN1)}}{=}
 	\min \{\|(\x, \js)\|_{f}, 1\} =\|(\x, \js)\|_{f},
 \end{align*}
so \eqref{qcP47-251206-6} is valid.
  If we take $\z=\kJX$ in \eqref{(FN4)}, then
\begin{align*}
	\|(\x, \js+ \jt)\|_{f} &\stackrel{\eqref{{a1}}}{=}
		\|(\x\opX \kJX, \js+ \jt)\|_{f}
\\&\geq \min \{\|(\xy, \js)\|_{f}, \|(\y\opX \kJX, \jt)\|_{f}\}
\\ &\stackrel{\eqref{{a1}}}{=}
    \min \{\|(\xy, \js)\|_{f}, \|(\y, \jt)\|_{f}\}
\end{align*}
which proves \eqref{qcP47-251206-7}.
	Suppose $\underset{\jt \to 0^+} \lim \|(\x, \jt)\|_{f} =1$.
If $\x\neq \kJX$, then $\|(\x, \jt_0)\|_{f}<1$ for some $\jt_0 >0$ by \eqref{(FN1)}.
It follows from \eqref{(FN2)} that
$\|(\x, \jt)\|_{f} \leq \|(\x, \jt_0)\|_{f} <1$ for all $\jt \leq \jt_0$.
$\underset{\jt \to 0^+} \lim \|(\x, \jt)\|_{f} \neq 1$, a contradiction.
So \eqref{qcP47-251206-8} is valid.
\end{proof}

\begin{remark}\label{RqP47-251206}
Every fuzzy normed BCI-algebra $(\Xo0,$ $\| \cdot \|_{f})$ also satisfies
 \eqref{qcP47-251206-1}, \eqref{qcP47-251206-1a}, \eqref{qcP47-251206-2}, \eqref{qcP47-251206-3},   \eqref{qcP47-251206-5}, \eqref{qcP47-251206-7}, and
 \eqref{qcP47-251206-8}.
\end{remark}

The example below shows that
any fuzzy normed BCI-algebra $(\Xo0,$ $\| \cdot \|_{f})$ does not satisfy
\eqref{qcP47-251206-6}.

\begin{example}\label{EqP47-251206B}
(i)	Let $\LomX=\{0,1,2\}$ be a BCI-algebra whose operation $\op$ is given by
\begin{align*}
 \begin{array}{c|ccc}
 	\op & \0 & \1 & \2\\
 	\hline
 	\0 & \0 & \1 & \0\\
 	\1 & \1 & \0 & \1\\
 	\2 & \2 & \1 & \0
 \end{array}
\end{align*}
	 Define a mapping $\|\cdot\|_f:\LomX\times(0,\infty)\to [0,1]$ by
$\|(\0,\jt)\|_f =1,$ $\|(\1,\jt)\|_f =1-e^{-\jt},$ and $\|(\2,\jt)\|_f =1-\tfrac12 e^{-\jt}$
for all $\jt>0$.
   Then $\|\cdot\|_{f}$ is a fuzzy BCI-norm on $\LomX$, but it does not satisfy \eqref{qcP47-251206-6}.
 Indeed, if we take $\x=\2$, $\z=\1$, $\js=0.1$, and $\jt=0.1$, then
 \begin{align*}
 \|(\xz, \js + \jt)\|_{f} &=\|(\2 \opX \1, 0.1+0.1)\|_{f} =\|(\1, 0.2)\|_{f}
  \\&=1-e^{-0.2}\approx 0.18127 \ngeq 0.54758 \approx 1-\tfrac{1}{2} e^{-0.1}
 \\& =\|(\2, 0.1)\|_{f} =\|(\x, \js)\|_{f}.
\end{align*}

(ii)	
	Let $\LomX=\{\0,\1,\2\}$ be a set with the binary operation $\op$ given by
\begin{align*}
	\begin{array}{c|ccc}
		\op & \0 & \1 & \2\\
		\hline
		\0 & \0 & \2 & \1\\
		\1 & \1 & \0 & \2\\
		\2 & \2 & \1 & \0
	\end{array}
\end{align*}
Then $\Xo0 :=(\LomX, \op, \0)$ is a BCI-algebra (see \cite{H-Book, MJ-Book}).
Define a mapping
\begin{align*}
	\| \cdot \|_{f} : \LomX\times (0, \infty)\rightarrow [0,1],
	~(\x, \jt)\mapsto \begin{cases}
		1   & \text{\rm if $\x=\0$},\\
		e^{-\tfrac{1}{\jt}} & \text{\rm if $\x \in \{\1, \2\}$}.
	\end{cases}
\end{align*}
It is routine to verify that $\| \cdot \|_{f}$ is a fuzzy BCI-norm on $\Xo0$,
but it does not satisfy \eqref{qcP47-251206-6} since
\begin{align*}
 \|(\1 \op \0, 2 +3)\|_{f} =\|(\1, 5)\|_{f} =e^{-\tfrac{1}{5}}<e^{-\tfrac{1}{2}}
 =\|(\1, 2)\|_{f}.
\end{align*}
\end{example}

\begin{proposition}\label{PqGV90-251107}
Every fuzzy normed BCK-algebra $(\Xo0,$ $\| \cdot \|_{f})$ satisfies
\begin{align}
	\label{qcGV90-251107}
\|(\xy, \js +\jt)\|_{f}\geq \min \{\|(\x, \js)\|_{f}, \|(\y, \jt)\|_{f}\},
\end{align}
for all $\x, \y \in \LomX$ and $\jt, \js \in (0, \infty)$.
\end{proposition}

\begin{proof}
Using \eqref{qcP47-251206-6} forces
\begin{align*}
	\|(\xy, \js +\jt)\|_{f}\geq \|(\x, \js)\|_{f}
	\geq \min \{\|(\x, \js)\|_{f}, \|(\y, \jt)\|_{f}\},
\end{align*}
so \eqref{qcGV90-251107} is valid.
\end{proof}

The example below shows that
any fuzzy normed BCI-algebra $(\Xo0,$ $\| \cdot \|_{f})$ does not satisfy
\eqref{qcGV90-251107}.

\begin{example}\label{EqGV20-251207}
Consider a BCI-algebra $\Xo0 :=(\LomX, \opX, 0)$ where $\LomX =\mathbb{Z}$ (integers) and the operation $\opX$ is defined  by
$\xy =\x -\y$ for all $\x, \y \in \LomX$.
Let $N : \LomX \rightarrow [0, \infty)$ be a mapping defined by
\begin{align*}
 N(\x)=\begin{cases}
 	|\x|   & \text{\rm if $\x \geq 0$},\\
 	4|\x|  & \text{\rm if $\x <0$}.
 \end{cases}
\end{align*}
Define a mapping
\begin{align*}
	\| \cdot \|_{f} : \LomX\times (0, \infty)\rightarrow [0,1],
	~(\x, \jt)\mapsto \tfrac{\jt}{N(\x)+\jt}.
\end{align*}
It is routine to verify that $(\Xo0, \| \cdot \|_{f})$ is a fuzzy normed BCI-algebra.
  		  Then
\begin{align*}
&\|(0\opX 1, 1+2)\|_{f}=\|(-1, 3)\|_{f}=\tfrac{3}{N(-1)+3}=\tfrac{3}{4+3}=\tfrac{3}{7},
\\& \|(0, 1)\|_{f}=\tfrac{1}{N(0)+1}=\tfrac{1}{0+1}=1,
\\& \|(1, 2)\|_{f}=\tfrac{2}{N(1)+2}=\tfrac{2}{1+2}=\tfrac{2}{3},
\end{align*}
so $\|(0\opX 1, 1+2)\|_{f}=\tfrac{3}{7} <\tfrac{2}{3}=\min \{1, \tfrac{2}{3}\}
=\min \{\|(0, 1)\|_{f}, \|(1, 2)\|_{f}\}.$
Hence \eqref{qcGV90-251107} is not valid.	
\end{example}

We have a question: If a mapping
$\| \cdot \|_{f} : \LomX \times (0, \infty) \rightarrow [0, 1]$ satisfies
\eqref{(FN1)}, \eqref{(FN2)}, \eqref{(FN3)}, and \eqref{qcGV90-251107}, then
is $(\LomX, \| \cdot \|_{f})$ a fuzzy normed BCK/BCI-algebra?
But the answer to this question is negative
because the condition \eqref{(FN4)} cannot be derived under the four conditions given.
An example given below illustrates this.

\begin{example}\label{EqGV20-250908}
Consider a BCK/BCI-algebra $\Xo0 :=(\LomX, \op, \0)$ where $\LomX =\{\0, \1, \2\}$ and the operation $\op$ is given by the following Cayley table:
	\begin{align*}\begin{array}{c|ccc}
			\opX & \0 & \1 & \2  \\
			\hline
			\0 & \0 & \0  & \0 \\
			\1 & \1 & \0  & \0 \\
			\2 & \2 & \1  & \0 \\
	\end{array}\end{align*}
	Define a mapping $\| \cdot \|_{f} : \LomX\times (0, \infty)\rightarrow [0,1]$ by
	\begin{align*}
		\| \cdot \|_{f}(\x, \jt)= \begin{cases}
			1   & \text{\rm if $\x=\0$},\\
			0.8   & \text{\rm if $\x=\1$},\\
			0.3   & \text{\rm if $\x=\2$}.\\
		\end{cases}
	\end{align*}
It is routine to confirm that $\| \cdot \|_{f}$ satisfies the conditions
\eqref{(FN1)}, \eqref{(FN2)}, \eqref{(FN3)}, and \eqref{qcGV90-251107}.
	But
\begin{align*}
& \|(\2\opX \0, \js +\jt)\|_{f} =\|(\2, \js +\jt)\|_{f}=0.3,
\\& \|(\2\opX \1, \js)\|_{f} =\|(\1, \js)\|_{f}=0.8,
\\& \|(\1\opX \0, \jt)\|_{f} =\|(\1, \jt)\|_{f}=0.8,
\end{align*}
which shows that
$\|(\2\opX \0, \js +\jt)\|_{f}\ngeq \min \{\|(\2\opX \1, \js)\|_{f}, \|(\1\opX \0, \jt)\|_{f}\},$
i.e., \eqref{(FN4)} is not valid. Hence
$(\LomX, \| \cdot \|_{f})$ is not a fuzzy normed BCK/BCI-algebra.
\end{example}

\begin{proposition}\label{PqP47-251207}
	Every fuzzy normed BCK/BCI-algebra $(\Xo0,$ $\| \cdot \|_{f})$ satisfies the following assertions:
	\begin{enumerate}
		\item[\rm (c1)]
		For each $\x\in \LomX$,
		the mapping $\jt \mapsto \|(\x, \jt)\|_{f}$ is non-decreasing on $(0, \infty)$.
		\item[\rm (c2)] If $\jt =\tfrac{\jt}{n} +\tfrac{(n-1)\jt}{n}$ for $n \in \mathbb{N}$,
		then
		\begin{align*}
			\x\leX \y ~\Rightarrow ~ \|(\x, \jt)\|_{f} \geq  \|(\y, \tfrac{(n-1)\jt}{n})\|_{f}.
		\end{align*}
	\end{enumerate}
\end{proposition}

\begin{proof}
	(c1) It is directly induced by \eqref{(FN1)} and \eqref{(FN2)}.
	
	(c2) Let $\x, \y\in \LomX$ be such that $\x\leX \y$, so $\xy =\kJX$, and $\jt =\tfrac{\jt}{n} +\tfrac{(n-1)\jt}{n}$
	for $n \in \mathbb{N}$. Then
	\begin{align*}
		\|(\x, \jt)\|_{f} &\stackrel{\eqref{{a1}}}{=}
		\|(\x\opX \kJX, \tfrac{\jt}{n} +\tfrac{(n-1)\jt}{n})\|_{f}
		\\& \stackrel{\eqref{(FN4)}}{\geq}
		\min \{\|(\xy, \tfrac{\jt}{n})\|_{f}, \|(\y\opX \kJX, \tfrac{(n-1)\jt}{n})\|_{f}\}
		\\& \stackrel{\eqref{{a1}}}{=}
		\min \{\|(\kJX, \tfrac{\jt}{n})\|_{f}, \|(\y, \tfrac{(n-1)\jt}{n})\|_{f}\}
		\\& \stackrel{\eqref{(FN1)}}{=}
		\min \{1, \|(\y, \tfrac{(n-1)\jt}{n})\|_{f}\}
		\\&=\|(\y, \tfrac{(n-1)\jt}{n})\|_{f}.
	\end{align*}
	This completes the proof.
\end{proof}

\begin{proposition}\label{PqGV70} {\rm (Chained triangle inequality)}
If  $(\Xo0,$ $\| \cdot \|_{f})$ is a fuzzy normed BCK/BCI-algebra, then
for every finite sequences $\x_1, \x_2, \cdots, \x_n$ in $\LomX$ $(n\geq 3)$
and $\jt_1, \jt_2, \cdots, \jt_{n-1} \in (0, \infty)$,
\begin{align*}
 &\|(\x_1\opX \x_n, \jt_1 +\jt_2 + \cdots +\jt_{n-1})\|_{f}
  \\&\geq \min \{\|(\x_1\opX \x_2, \jt_1)\|_{f}, \|(\x_2\opX \x_3, \jt_2)\|_{f}, \cdots,
      \|(\x_{n-1}\opX \x_n, \jt_{n-1})\|_{f} \}.
\end{align*}
\end{proposition}

\begin{proof}
The proof will be proceeded by induction on $n$.
For $n=3$, it is direct from \eqref{(FN4)}. Suppose it holds for $k=n-1$, that is,
\begin{align*}
	&\|(\x_1\opX \x_{n-1}, \jt_1 +\jt_2 + \cdots +\jt_{n-2})\|_{f}
	\\&\geq \min \{\|(\x_1\opX \x_2, \jt_1)\|_{f}, \|(\x_2\opX \x_3, \jt_2)\|_{f}, \cdots,
	\|(\x_{n-2}\opX \x_{n-1}, \jt_{n-2})\|_{f} \}.
\end{align*}
It follows from the induction hypothesis that
\begin{align*}
	&\|(\x_1\opX \x_n, (\jt_1 +\jt_2 + \cdots +\jt_{n-2})+\jt_{n-1})\|_{f}
  \\&\geq \min \{\|(\x_1\opX \x_{n-1}, \jt_1 +\jt_2 + \cdots +\jt_{n-2})\|_{f},
        \|(\x_{n-1}\opX \x_n, \jt_{n-1})\|_{f}\}
	\\&\geq \min \{\|(\x_1\opX \x_2, \jt_1)\|_{f}, \|(\x_2\opX \x_3, \jt_2)\|_{f}, \cdots,
	\|(\x_{n-1}\opX \x_n, \jt_{n-1})\|_{f} \}.
\end{align*}
This completes the proof.
\end{proof}

We have the following  questions:
\begin{enumerate}
	\item[\rm (Q1)] If $\x\leX \y$, then $\|(\x, \jt)\|_{f}\geq \|(\y, \jt)\|_{f}$
for all $\jt \in (0, \infty)$?
	\item[\rm (Q2)] If $\x\leX \y$, then  $\|(\x, \jt)\|_{f}\leq \|(\y, \jt)\|_{f}$
	for all $\jt \in (0, \infty)$?
\end{enumerate}

In the following example, we provide the negative answer to (Q1) and (Q2).

\begin{example}\label{Eq[A3-2]}
(i) Consider a BCK/BCI-algebra $\Xo0 :=(\LomX, \op, \0)$ where $\LomX =\{\0, \1, \2\}$ and the operation $\op$ is given by the following Cayley table:
\begin{align*}\begin{array}{c|ccc}
		\op & \0 & \1 & \2  \\
		\hline
		\0 & \0 & \0  & \0 \\
		\1 & \1 & \0  & \0 \\
		\2 & \2 & \2  & \0 \\
\end{array}\end{align*}
Define a mapping $\| \cdot \|_{f} : \LomX\times (0, \infty)\rightarrow [0,1]$ by
\begin{align*}
	\| \cdot \|_{f}(\x, \jt)= \begin{cases}
		1   & \text{\rm if $\x=0$},\\
		0.5 & \text{\rm if $\x =\1$ and $\jt\leq 1$}, \\
		0.7 +0.3(1-\tfrac{1}{e^{t-1}}) & \text{\rm if $\x =\1$ and $\jt > 1$}, \\
		0.5 & \text{\rm if $\x =\2$ and $\jt < 1$}, \\
		0.6 & \text{\rm if $\x =\2$ and $\jt = 1$}, \\
		0.7 +0.3(1-\tfrac{1}{e^{t-1}}) & \text{\rm if $\x =\2$ and $\jt > 1$}.
	\end{cases}
\end{align*}
It is routine to confirm that $(\Xo0,$ $\| \cdot \|_{f})$ is a fuzzy normed BCK/BCI-algebra.
Note that $\1\leX \2$, but if $\jt =1$, then $\| \cdot \|_{f}(\1, 1)=0.5<0.6=\| \cdot \|_{f}(\2, 1)$.
Hence $\| \cdot \|_{f}(\x, \jt)\ngeq \| \cdot \|_{f}(\y, \jt)$ for $\x=\1,$ $\y=\2$ and $\jt =1$.

(ii) Consider the fuzzy BCI-norm $\| \cdot \|_{f}$ in Example \ref{Eq[A2-1]}(i).
Note that $\0\leX \1$, but
$\| (\0, \jt) \|_{f} =1 >1-\tfrac{1}{e^{t}}=\| (\1, \jt) \|_{f}$,
which shows that the answer to (Q2) is negative.
\end{example}

\begin{question}\label{QqGV90-251208}
	Let \kXo0 and \kYo0  be BCK/BCI-algebras, and let
$\kbp : \LomX \rightarrow \LomY$ be a  homomorphism.  	
Consider two mappings  $\| \cdot \|_{f}: \LomX\times (0, \infty)\rightarrow [0,1]$ and
$\| \cdot \|_{g} : \LomY\times (0, \infty)\rightarrow [0,1]$ such that
 $\| \cdot \|_{f} =\| \cdot \|_{g} \circ \kbp$, i.e.,
 $\|(\x, \jt)\|_{f} =\|(\kbp(\x), \jt)\|_{g}$ for all $\x\in \LomX$ and $\jt >0$.

(i) If $\| \cdot \|_{f}$ is a fuzzy BCK/BCI-norm on $\Xo0$, then is
$\| \cdot \|_{g}$  a fuzzy BCK/BCI-norm on $\Yo0$?

(ii) If $\| \cdot \|_{g}$ is a fuzzy BCK/BCI-norm on $\Yo0$, then is
$\| \cdot \|_{f}$  a fuzzy BCK/BCI-norm on $\Xo0$?
\end{question}

The examples below give a negative answer to the Question \ref{QqGV90-251208}.

\begin{example}\label{EqE38-251208}
 (i) Consider two BCK-algebras $\Xo0$ $:=$ $(\LomX,$ $\opX,$ $\0)$ and
$\Yo0$ $:=$ $(\LomY,$ $\opY,$ $\0)$
where $\LomX =\{\0, \1\}$ and $\LomY =\{\0, \2, \3\}$, and
the operations $\opX$ and $\opY$ are  given by the following Cayley tables:
\begin{align*}
	\begin{array}{c|cc}
		\opX & \0 & \1  \\
		\hline
		\0 & \0 & \0  \\
		\1 & \1 & \0  \\
\end{array}
\hspace{10mm}
	\begin{array}{c|ccc}
	\opY & \0 & \2  & \3 \\
	\hline
	\0 & \0 & \0 & \0  \\
	\2 & \2 & \0 & \0  \\
	\3 & \3 & \2 & \0  \\
\end{array}
\end{align*}
Define a mapping $\kbp : \LomX \rightarrow \LomY$ by $\kbp(\0)=\0$ and $\kbp(\1)=\2$.
	Then $\kbp$ is a homomorphism.
	Define $\| \cdot \|_{f}: \LomX\times (0, \infty)\rightarrow [0,1]$ by
$\|(\0, \jt)\|_{f} =1$ and $\|(\1, \jt)\|_{f} =e^{-\tfrac{1}{\jt}}$ for every $\jt >0$.
	Then it is a fuzzy BCK-norm on $\Xo0$.
	If we give a mapping $\| \cdot \|_{g}: \LomY\times (0, \infty)\rightarrow [0,1]$ defined by
 $\|(\0, \jt)\|_{g}=1=\|(\3, \jt)\|_{g}$ and $\|(\2, \jt)\|_{g}=e^{-\tfrac{1}{\jt}}$
 for every $\jt >0$, then
$\|(\x, \jt)\|_{f} =\|(\kbp(\x), \jt)\|_{g}$ for all $\x\in \LomX$ and $\jt >0$.
But $\| \cdot \|_{g}$ is not a fuzzy BCK-norm on $\Yo0$ since
 $\|(\3, \jt)\|_{g}=1$ for all $\jt>0$ even though $\3 \neq \0$.

(ii)  Consider a BCI-algebra $\Xo0 :=(\LomX, \opX, 0)$ where $\LomX =\mathbb{Z}$ (integers) and the operation $\opX$ is defined  by
$\xy =\x -\y$ for all $\x, \y \in \LomX$, and let
$\Yo0 :=(\LomY =\{\kJY\}, \opX, \kJY)$ be the trivial BCI-algebra.
Then the zero mapping $\kbp : \LomX \rightarrow \LomY, ~\x\mapsto \kJX,$ is a homomorphism,
and  $\| \cdot \|_{g} : \LomY\times (0, \infty)\rightarrow [0,1], ~(\kJY, \jt)\mapsto 1,$
is a fuzzy BCI-norm on $\Yo0$.
Define a mapping $\| \cdot \|_{f}: \LomX\times (0, \infty)\rightarrow [0,1]$ by
$\|(\x, \jt)\|_{f} =\|(\kbp(\x), \jt)\|_{g}$ for all $\x\in \LomX$ and $\jt>0$.
Then $\|(7, \jt)\|_{f} =\|(\kbp(7), \jt)\|_{g}=\|(\kJY, \jt)\|_{g}=1$ and $7\neq 0$.
Hence $\| \cdot \|_{f}$ is not a fuzzy BCI-norm on $\Xo0$.
\end{example}

By strengthening the condition of homomorphism $\kbp$, we consider the following theorems.

\begin{theorem}\label{TqT39-251207}
	Let \kXo0 and \kYo0  be BCK/BCI-algebras, and let
$\kbp : \LomX \rightarrow \LomY$ be a  homomorphism.  	
Consider two mappings  $\| \cdot \|_{f}: \LomX\times (0, \infty)\rightarrow [0,1]$ and
$\| \cdot \|_{g} : \LomY\times (0, \infty)\rightarrow [0,1]$ such that
$\| \cdot \|_{f} =\| \cdot \|_{g} \circ \kbp$, i.e.,
$\|(\x, \jt)\|_{f} =\|(\kbp(\x), \jt)\|_{g}$ for all $\x\in \LomX$ and $\jt >0$.
\begin{enumerate}
	\item[\rm (i)]
If $\kbp$ is surjective and $\| \cdot \|_{f}$ is a fuzzy BCK/BCI-norm on $\Xo0$, then
   $\| \cdot \|_{g}$ is a fuzzy BCK/BCI-norm on $\Yo0$.
\item[\rm (ii)] If $\kbp$ satisfies
	\begin{align}\label{qcT39-251207}
	(\forall \x\in \LomX) (\kbp(\x)=\kJY ~\Rightarrow ~\x=\kJX)
	\end{align}
	and $\| \cdot \|_{g}$ is a fuzzy BCK/BCI-norm on $\Yo0$, then
	$\| \cdot \|_{f}$ is a fuzzy BCK/BCI-norm on $\Xo0$.
\item[\rm (iii)] If $\kbp$ is injective
and $\| \cdot \|_{g}$ is a fuzzy BCK/BCI-norm on $\Yo0$, then
$\| \cdot \|_{f}$ is a fuzzy BCK/BCI-norm on $\Xo0$.
\end{enumerate}
\end{theorem}

\begin{proof}
(i) Let $\kbp$ be surjective and $\| \cdot \|_{f}$ be a fuzzy BCK/BCI-norm on $\Xo0$.
	Suppose $\|(\y, \jt)\|_{g}=1$ for all $\y\in \LomY$ and $\jt >0$.
	Then there exists $\x\in \LomX$ such that $\kbp(\x)=\y$, so
$1=\|(\y, \jt)\|_{g}=\|(\kbp(\x), \jt)\|_{g}=\|(\x, \jt)\|_{f}.$
It follows from the \eqref{(FN1)} for $\| \cdot \|_{f}$ that $\x=\kJX$.
Thus $\y=\kbp(\x)=\kbp(\kJX)=\kJY$.
Conversely, if $\y=\kJY$, then
$\|(\y, \jt)\|_{g}=\|(\kJY, \jt)\|_{g}=\|(\kbp(\kJX), \jt)\|_{g}
=\|(\kJX, \jt)\|_{f}=1$ by the \eqref{(FN1)} for $\| \cdot \|_{f}$.
Hence $\| \cdot \|_{g}$ satisfies \eqref{(FN1)}.
	Let $\y\in \LomY$ and $\jt \geq \js$ in $(0, \infty)$.
Then $\kbp(\x)=\y$ for some $\x\in \LomX$, so
\begin{align*}
 \|(\y, \jt)\|_{g}=\|(\kbp(\x), \jt)\|_{g}=\|(\x, \jt)\|_{f} \geq \|(\x, \js)\|_{f}
 =\|(\kbp(\x), \js)\|_{g}=\|(\y, \js)\|_{g}
\end{align*}
by the \eqref{(FN2)} for $\| \cdot \|_{f}$. Thus  $\| \cdot \|_{g}$ satisfies \eqref{(FN2)}.
	For every $\y\in \LomY$ and $\jt \in (0, \infty)$, if we take $\x\in \LomX$ with $\kbp(\x)=\y$,
	then
\begin{align*}
 \underset{\jt \to \infty} \lim \|(\y, \jt)\|_{g}
 =\underset{\jt \to \infty} \lim \|(\kbp(\x), \jt)\|_{g}
= \underset{\jt \to \infty} \lim \|(\x, \jt)\|_{f}=1
\end{align*}
by the \eqref{(FN3)} for $\| \cdot \|_{f}$, which shows that
$\| \cdot \|_{g}$ satisfies \eqref{(FN3)}.
	Let $\x, \y, \z \in \LomY$ and $\js, \jt \in (0,\infty)$. Then
	there are $\kx, \ky, \kz \in \LomX$ such that $\kbp(\kx)=\x,$ $\kbp(\ky)=\y$ and $\kbp(\kz)=\z$.
Using the  \eqref{(FN4)} for $\| \cdot \|_{f}$ and the fact that $\kbp$ is a homomorphism induces
\begin{align*}
\|(\x \opY z, \js +\jt)\|_{g}&=\|(\kbp(\kx)\opY \kbp(\kz), \js +\jt)\|_{g}
 \\&=\|(\kbp(\kx \opX \kz), \js +\jt)\|_{g}
 \\&=\|(\kx \opX \kz, \js +\jt)\|_{f}
 \\&\geq \min \{\|(\kx \opX \ky, \js)\|_{f}, \|(\ky \opX \kz, \jt)\|_{f}\}
 \\&= \min \{\|(\kbp(\kx \opX \ky), \js)\|_{g}, \|(\kbp(\ky \opX \kz), \jt)\|_{g}\}
 \\&= \min \{\|(\kbp(\kx) \opY \kbp(\ky)), \js)\|_{g}, \|(\kbp(\ky) \opY \kbp(\kz)), \jt)\|_{g}\}
 \\&= \min \{\|(\x \opY \y, \js)\|_{g}, \|(\y \opY \z, \jt)\|_{g}\},
\end{align*}
so $\| \cdot \|_{g}$ satisfies \eqref{(FN4)}. Therefore
$\| \cdot \|_{g}$ is a fuzzy BCK/BCI-norm on $\Yo0$.

(ii) Suppose $\kbp$ satisfies \eqref{qcT39-251207} and let
$\| \cdot \|_{g}$ be a fuzzy BCK/BCI-norm on $\Yo0$.
	Assume that $\|(\x, \jt)\|_{f}=1$ for all $\x\in \LomX$ and $\jt \in (0, \infty)$.
Then
 $\|(\kbp(\x), \jt)\|_{g}=\|(\x, \jt)\|_{f}=1$, so
 $\kbp(\x)=\kJY$ by the \eqref{(FN1)} for $\| \cdot \|_{g}$.
Hence $\x=\kJX$ by \eqref{qcT39-251207}.
Conversely, if $\x=\kJX$, then
$\|(\x, \jt)\|_{f}=\|(\kJX, \jt)\|_{f}=\|(\kbp(\kJX), \jt)\|_{g}
 =\|(\kJY, \jt)\|_{g}=1$
by the \eqref{(FN1)} for $\| \cdot \|_{g}$. So $\| \cdot \|_{f}$ satisfies \eqref{(FN1)}.
	For every $\x\in \LomX$ and $\js, \jt \in (0, \infty)$ with $\jt \geq \js$, we have
\begin{align*}
\|(\x, \js)\|_{f} =\|(\kbp(\x), \js)\|_{g}\leq \|(\kbp(\x), \jt)\|_{g}=\|(\x, \js)\|_{f}
\end{align*}	
by the \eqref{(FN2)} for $\| \cdot \|_{g}$, and
\begin{align*}
\underset{\jt \to \infty}\lim \|(\x, \jt)\|_{f}
  =\underset{\jt \to \infty}\lim \|(\kbp(\x), \jt)\|_{g}=1
\end{align*}
by the \eqref{(FN3)} for $\| \cdot \|_{g}$.
 Let $\x, \y, \z\in \LomX$ and $\js, \jt \in (0, \infty)$. Then
\begin{align*}
\|(\xz, \js +\jt)\|_{f}&=\|(\kbp(\xz), \js +\jt)\|_{g}
  =\|(\kbp(\x)\opY \kbp(\z), \js +\jt)\|_{g}
  \\& \geq \min \{\|(\kbp(\x)\opY \kbp(\y), \js)\|_{g}, \|(\kbp(\y)\opY \kbp(\z), \jt)\|_{g}\}
  \\& \geq \min \{\|(\kbp(\xy), \js)\|_{g}, \|(\kbp(\yz), \jt)\|_{g}\}
\\&=\min \{\|(\xy, \js)\|_{f}, \|(\yz, \jt)\|_{f}\}
\end{align*}
by the \eqref{(FN4)} for $\| \cdot \|_{g}$ and the fact that $\kbp$ is a homomorphism.
Consequently, $\| \cdot \|_{f}$ is a fuzzy BCK/BCI-norm on $\Xo0$.

(iii) Assume that  $\kbp$ is injective
and $\| \cdot \|_{g}$ is a fuzzy BCK/BCI-norm on $\Yo0$.
	Let $\x\in \LomX$ and   $\jt \in (0, \infty)$.
If $\|(\x, \jt)\|_{f}=1$, then $\|(\kbp(\x), \jt)\|_{g}=1$ so $\kbp(\x)=\kJY$ by the \eqref{(FN1)} for
$\| \cdot \|_{g}$. Since $\kbp$ is injective, it follows that $\x=\kJX$.
If $\x=\kJX$, then $\|(\x, \jt)\|_{f}=\|(\kJX, \jt)\|_{f}=\|(\kbp(\kJX), \jt)\|_{g}
=\|(\kJY, \jt)\|_{g}=1$
by the \eqref{(FN1)} for $\| \cdot \|_{g}$ and the fact that $\kbp$ is a homomorphism.
 So $\| \cdot \|_{f}$ satisfies \eqref{(FN1)}.
Hence $\| \cdot \|_{f}$ satisfies \eqref{(FN1)}.
We can see that  $\| \cdot \|_{f}$ satisfies \eqref{(FN2)}, \eqref{(FN3)} and \eqref{(FN4)}
in a similar way to the proof of (ii).
Therefore $\| \cdot \|_{f}$ is a fuzzy BCK/BCI-norm on $\Xo0$.
\end{proof}

\begin{theorem}\label{TqT314-251207}
Every BCK/BCI-algebra \kXo0 admits a fuzzy normed BCK/BCI-algebra $(\Xo0, \|\cdot \|_{f})$.
\end{theorem}

\begin{proof}
Define a mapping
\begin{align*}
	\| \cdot \|_{f} : \LomX\times (0, \infty)\rightarrow [0,1],
	~(\x, \jt)\mapsto \begin{cases}
		1   & \text{\rm if $\x=\kJX$},\\
		\tfrac{\jt}{\jt +1} & \text{\rm if $\x \neq \kJX$}.
	\end{cases}
\end{align*}
It is clear that for all $\x\in \LomX$ and $\jt \in (0, \infty)$,
 $\|(\x, \jt)\|_{f}=1$ if and only if $\x=\kJX$.
	Let $\x\in \LomX$ and $\jt \geq \js$ in $(0, \infty)$.
If $\x=\kJX$, then $\|(\x, \jt)\|_{f} =1=\|(\x, \js)\|_{f}$.
If $\x\neq \kJX$, then 	
	$\|(\x, \jt)\|_{f}=\tfrac{\jt}{\jt +1} \geq \tfrac{\js}{\js +1}=\|(\x, \js)\|_{f}$.
Thus 	$\| \cdot \|_{f}$ satisfies \eqref{(FN2)}.
	It is clear that if $\x=\kJX$, then
 $\underset{\jt \to \infty} \lim \|(\x, \jt)\|_{f} =1$. 	
If $\x \neq \kJX$, then 	
\begin{align*}
\underset{\jt \to \infty} \lim \|(\x, \jt)\|_{f}
   =\underset{\jt \to \infty} \lim \tfrac{\jt}{\jt +1}
   =\underset{\jt \to \infty} \lim \tfrac{1}{1+\tfrac{1}{\jt}} =1.
   \end{align*}
Hence   	$\| \cdot \|_{f}$ satisfies \eqref{(FN3)}.
 Let $\x, \y, \z\in \LomX$ and $\js, \jt \in (0, \infty)$. Then
If $\xz =0$, then 	
$\|(\xz, \js + \jt)\|_{f} =1 \geq  \min \{\|(\xy, \js)\|_{f}, \|(\yz, \jt)\|_{f}\}$
Now we consider two cases:
\begin{center}
(i) $\xy =\kJX$ and $\yz=\kJX$;
(ii)   $\xy \neq \kJX$ or  $\yz \neq \kJX$.
\end{center}
For the first case, we have $\|(\xy, \js)\|_{f}=1$, $\|(\yz, \jt)\|_{f}=1$, and
$\xz =\kJX$ since the induced order $\leX$ is transitive.  Hence
\begin{align*}
\|(\xz, \js + \jt)\|_{f} =1 =\min \{1,1\}=\min \{\|(\xy, \js)\|_{f}, \|(\yz, \jt)\|_{f}\}.
\end{align*}
The second case implies $\|(\xy, \js)\|_{f}=\tfrac{\js}{\js +1}$ or
   $\|(\yz, \jt)\|_{f}=\tfrac{\jt}{\jt +1}$.
It is obvious that if $\xz =\kJX$, then
\begin{align*}
\|(\xz, \js + \jt)\|_{f}=1 \geq \min \{\|(\xy, \js)\|_{f}, \|(\yz, \jt)\|_{f}\}.
\end{align*}
If $\xz \neq \kJX$, then
\begin{align*}
\|(\xz, \js + \jt)\|_{f}&=\tfrac{\js +\jt}{\js +\jt +1}
 \geq \min \{\tfrac{\js}{\js +1}, \tfrac{\jt}{\jt +1}\}
 \\&=\min \{\|(\xy, \js)\|_{f}, \|(\yz, \jt)\|_{f}\}.
\end{align*}
Therefore $(\Xo0, \|\cdot \|_{f})$ is a fuzzy normed BCK/BCI-algebra. 	
\end{proof}

\begin{theorem}
	Let \kXo0 and \kYo0  be BCK/BCI-algebras, and let
$\kbp : \LomX \rightarrow \LomY$ be a  homomorphism.  	
If $\kbp$ is a bijection and  $(\Xo0,$ $\| \cdot \|_{f})$ is a fuzzy normed BCK/BCI-algebra, then
$(\Yo0,$ $\| \cdot \|_{g})$ is a fuzzy normed BCK/BCI-algebra where
$\|(\y, \jt)\|_{g} =\|(\kbp^{-1}(\y), \jt) \|_{f}$ for all $\y\in \LomY$ and $\jt>0$.
\end{theorem}

\begin{proof}
Let $\x, \y, \z\in \LomY$ and $\js, \jt \in (0, \infty)$.  Since $\kbp$ is  a bijection,
there exist unique elements $\kx, \ky, \kz \in \LomX$ such that $\kbp(\kx)=\x,$ $\kbp(\ky)=\y,$ and
$\kbp(\kz)=\z$, equivalently, $\kx =\kbp^{-1}(\x),$ $\ky =\kbp^{-1}(\y),$ and $\kz =\kbp^{-1}(\z)$.
	Then
 \begin{align*}
\|(\x, \jt)\|_{g}=1 &~\Leftrightarrow ~\|(\kbp^{-1}(\x), \jt) \|_{f}=1
 \\&~\Leftrightarrow ~\|(\kx, \jt) \|_{f}=1
 \\& ~\Leftrightarrow ~\kx =\kJX  \qquad \text{\rm (by (FN1) for $\| \cdot \|_{f}$)}
 \\& ~\Leftrightarrow ~\kbp(\kx) =\kbp(\kJX)
 \\& ~\Leftrightarrow ~\x=\kJY,
  \end{align*}
so $\| \cdot \|_{g}$ satisfies \eqref{(FN1)}.
	Suppose $\jt \geq \js$. Then
\begin{align*}
 \|(\x, \jt)\|_{g}&=\|(\kbp^{-1}(\x), \jt) \|_{f} =\|(\kx, \jt) \|_{f}
 \\&\geq \|(\kx, \js) \|_{f} \qquad \text{\rm (by (FN2) for $\| \cdot \|_{f}$)}
\\&=\|(\kbp^{-1}(\x), \js) \|_{f}
\\&= \|(\x, \js)\|_{g}
\end{align*}	
and
$\underset{\jt \to \infty} \lim \|(\x, \jt)\|_{g}
 =\underset{\jt \to \infty} \lim \|(\kbp^{-1}(\x), \jt) \|_{f}
 =\underset{\jt \to \infty} \lim \|(\kx, \jt) \|_{f}=1$
by (FN3) for $\| \cdot \|_{f}$. Thus
 $\| \cdot \|_{g}$ satisfies \eqref{(FN2)} and \eqref{(FN3)}.
 Since $\kbp$ is a bijection, $\kbp^{-1}$ is also a homomorphism from $\LomY$ to $\LomX$.
Hence
\begin{align*}
&\|(\x\opY \z, \js +\jt)\|_{g} =\|(\kbp^{-1}(\x \opY \z), \js +\jt) \|_{f}
\\&=\|(\kbp^{-1}(\x) \opX \kbp^{-1}(\z), \js +\jt) \|_{f}
\\&=\|(\kx \opX \kz, \js +\jt) \|_{f}
 \\&  \geq \min \{\|(\kx \opX \ky, \js) \|_{f}, \|(\ky \opX \kz, \jt) \|_{f}\}
\\&=\min \{\|(\kbp^{-1}(\x) \opX \kbp^{-1}(\y), \js) \|_{f},
           \|(\kbp^{-1}(\y) \opX \kbp^{-1}(\z), \jt) \|_{f}\}
\\&=\min \{\|(\kbp^{-1}(\x \opY \y), \js) \|_{f},
    \|(\kbp^{-1}(\y \opY \z), \jt) \|_{f}\}
\\&=\min \{\|(\x\opY \y, \js)\|_{g}, \|(\y\opY \z, \jt)\|_{g}\}
\end{align*}
by (FN4) for $\| \cdot \|_{f}$.
Consequently, $(\Yo0,$ $\| \cdot \|_{g})$ is a fuzzy normed BCK/BCI-algebra.
\end{proof}

\begin{theorem}\label{PqQ27-251123}
	Let $\| \cdot \|_{f} : \LomX \times (0, \infty) \rightarrow [0, 1]$ be a
	mapping satisfying the conditions \eqref{(FN1)}, \eqref{(FN2)}, and \eqref{(FN3)}. Then
	$(\Xo0, \| \cdot \|_{\kwp})$ is a fuzzy normed BCK/BCI-algebra  if and only if it satisfies
	\begin{align}\label{qcQ27-251123}
		\|(\x, \js +\jt)\|_{\kwp}\geq \min \{\|(\xy, \js)\|_{\kwp}, \|(\y, \jt)\|_{\kwp}\}
	\end{align}
	for all $\x, \y\in \LomX$ and $\js, \jt \in (0, \infty)$.
\end{theorem}

\begin{proof}
    	Suppose $(\Xo0, \| \cdot \|_{\kwp})$ is a fuzzy normed BCK/BCI-algebra.
	If we take $\z=\kJX$ in \eqref{(FN4)}, then
	\begin{align*}
		\|(\x, \js+ \jt)\|_{\kwp}&\stackrel{\eqref{{a1}}}{=}
		\|(\x\opX \kJX, \js+ \jt)\|_{\kwp}
		\\&\geq \min \{\|(\xy, \js)\|_{\kwp}, \|(\y\opX \kJX, \jt)\|_{\kwp}\}
		\\&\stackrel{\eqref{{a1}}}{=}
		\min \{\|(\xy, \js)\|_{\kwp}, \|(\y, \jt)\|_{\kwp}\},
	\end{align*}
	which proves \eqref{qcQ27-251123}.
	
	Conversely, assume that $(\Xo0, \| \cdot \|_{\kwp})$ satisfies \eqref{qcQ27-251123}.
	We first show that
	\begin{align}\label{qcQ27-251123-1}
		\x\leX \y \implies \|(\x, \jt)\|_{\kwp} \geq \|(\y, \jt)\|_{\kwp}
	\end{align}
	for all $\x, \y\in \LomX$ and $\jt \in (0, \infty)$.
	Let $\x, \y\in \LomX$ be such that $\x\leX \y$. Then $\xy =\kJX$, so
	\begin{align*}
		\|(\x, \js +\jt)\|_{\kwp} &\geq \min \{\|(\xy, \js)\|_{\kwp}, \|(\y, \jt)\|_{\kwp}\}
		\\&=\min \{\|(\kJX, \js)\|_{\kwp}, \|(\y, \jt)\|_{\kwp}\}
		\\&\stackrel{\eqref{(FN1)}}{=}
		\min \{1, \|(\y, \jt)\|_{\kwp}\}=\|(\y, \jt)\|_{\kwp}
	\end{align*}
	Since $\js$ can be arbitrarily small (approach to $0$),
	we can choose $\js \in (0, \infty)$ such that $\js+\jt$ is arbitrarily close to $\jt$.
	Thus we can obtain $\|(\x, \jt)\|_{\kwp} \geq \|(\y, \jt)\|_{\kwp}$.
	Since $(\xz)\opX (\yz)\leX \xy$ for all $\x, \y, \z\in \LomX$,
	it follows from  \eqref{qcQ27-251123-1} that
	\begin{align*}
		\|((\xz)\opX (\yz), \js)\|_{\kwp} \geq \|(\xy, \js)\|_{\kwp}.
	\end{align*}
	Using \eqref{qcQ27-251123}, we have
	\begin{align*}
		\|(\xz, \js +\jt)\|_{\kwp}&\geq
		\min \{\|((\xz)\opX (\yz), \js)\|_{\kwp}, \|(\yz, \jt)\|_{\kwp}\}
		\\&\geq \min \{\|(\xy, \js)\|_{\kwp}, \|(\yz, \jt)\|_{\kwp}\},
	\end{align*}
	so \eqref{(FN4)} is valid. Therefore
	$(\Xo0, \| \cdot \|_{\kwp})$ is a fuzzy normed BCK/BCI-algebra.
\end{proof}

%%%%%%%%%%%%%%%%%%%%%%%%%%%%%%%%%%%%%%%%%%%%%%%%%%%%%%%%%%%%%%%%%%%%%%
\section{Conclusions}
%%%%%%%%%%%%%%%%%%%%%%%%%%%%%%%%%%%%%%%%%%%%%%%%%%%%%%%%%%%%%%%%%%%%%%

In this paper, we have introduced the concept of fuzzy normed
BCK-algebras and fuzzy normed BCI-algebras and systematically developed
their fundamental properties. Various examples were provided to
illustrate the independence of the axioms and to clarify the differences
between the BCK and BCI cases. We established several inequalities and
order-related properties that characterize fuzzy norms in these
algebras.

A major contribution of this work is the analysis of fuzzy norms under
algebra homomorphisms, including injective, surjective, and bijective
cases. We also proved that every BCK/BCI-algebra admits a fuzzy norm and
obtained a characterization theorem reducing the main fuzzy norm
condition to a simpler inequality.

These results enrich the theory of BCK/BCI-algebras by introducing
fuzzy-analytical tools and open several directions for future research,
such as fuzzy normed ideals and filters, topological structures induced
by fuzzy norms, and applications to logical systems and decision-making
models.

%%%%%%%%%%%%%%%%%%%%%%%%%%%%%%%%%%%%%%%%%%%%%%%%%%%%%%%%%%%%%%%%%%%%%%%

\end{document}